\documentclass[12pt]{amsart}
\usepackage{graphicx}
\usepackage{amssymb}
\usepackage{amsfonts}
\usepackage{amsmath,amscd}
\usepackage{array}
\usepackage{rotating}
\usepackage{epsfig}
\usepackage[matrix,frame]{xypic}

\newtheorem{example}{Example}[section]

\newtheorem{theorem}[example]{Theorem}

\newtheorem{definition}[example]{Definition}

\newtheorem{lemma}[example]{Lemma}

\newcommand{\sous}{\curvearrowleft}

\def\Id{{\rm Id}}
\def\Sym{{\bf Sym}}            
\def\FQSym{{\bf FQSym}}        
\def\PBT{{\bf PBT}}            

\def\CBT{{\bf CBT}}
\def\maj{{\rm imaj\,}}
\def\maj{{\rm maj\,}}

\def\P{{\bf P}}

\def\R{{\mathbb R}}
\def\N{{\mathbb N}}

\def\<{\langle}
\def\>{\rangle}

\def\F{{\bf F}}         
\def\S{{\bf S}}         
\def\SG{{\mathfrak S}}  

\def\ie{{\it i.e.}}

\def\maj{{\rm maj\,}}

\def\LL{{\mathcal L}}
\def\GG{{\mathcal G}}
\def\ad{{\rm ad\,}}

\def\H{{\mathcal H}}
\def\shuff#1#2{\mathbin{
\hbox{\vbox{ \hbox{\vrule \hskip#2 \vrule height#1 width 0pt
}%
\hrule}%
\vbox{ \hbox{\vrule \hskip#2 \vrule height#1 width 0pt
\vrule }%
\hrule}%
}}}

\long\def\psboxit#1#2{%
\begingroup\setbox0=\hbox{#2}%
\dimen0=\ht0 \advance\dimen0 by \dp0%
    \hbox{%
    \copy0%
    }
\endgroup%
}
\def\SetTableau#1#2#3#4{%
  \gdef\Tabvrule{\vrule\vrule width-0.4pt}
  \gdef\Tabhrule{\hrule\hrule height-0.4pt}  
  \gdef\Tabstrut{\vrule height#1 depth#2 width0pt\relax}
  \gdef\Tabbox##1{\hbox to #3{\hskip0.4pt\hfill\Tabstrut$#4##1$\hfill}}
} 
\def\Case#1{\vcenter{\Tabhrule%
                   \hbox{\Tabvrule\Tabbox{#1}\Tabvrule}\Tabhrule}}

\def\GenTab#1{\vcenter{\halign{&$\Case{##}$\cr#1}}\egroup}
\def\Tableau{%
  \bgroup%
  \let\ =\omit%
  \let\\=\cr%
  \offinterlineskip\GenTab}

\def\qbin#1#2{\begin{bmatrix} #1 \\ #2\end{bmatrix}}

\def\shuf{{\mathchoice{\shuff{7pt}{3.5pt}}%
{\shuff{6pt}{3pt}}%
{\shuff{4pt}{2pt}}%
{\shuff{3pt}{1.5pt}}}}%
\def\shuffle{\,\shuf\,}

\title[An operational calculus for the {\bf Mould} operad]{An operational calculus for the ${\bf Mould}$ operad}

\author[F.~Chapoton, F.~Hivert, J.-C.~Novelli, and J.-Y.~Thibon]{Fr\'ed\'eric Chapoton, 
        Florent Hivert,\\ 
        Jean-Christophe Novelli,\\ 
        and Jean-Yves Thibon}

\address[F. Chapoton]{Institut Camille Jordan, Universit\'e Claude Bernard Lyon 1, F-69622 Villeurbanne Cedex, FRANCE}
\address[F. Hivert]{LIFAR, Universit\'e de Rouen, 76801 Saint-Etienne-du-Rouvray Cedex, FRANCE}
\address[J.-C. Novelli, J.-Y. Thibon]{Institut Gaspard Monge, Universit\'e Paris-Est, 77454 Marne-la-Vall\'ee Cedex 2, FRANCE}
\email[Fr\'ed\'eric Chapoton]{chapoton@math.univ-lyon1.fr}
\email[Florent Hivert]{Florent.Hivert@univ-rouen.fr}
\email[Jean-Christophe Novelli]{novelli@univ-mlv.fr}
\email[Jean-Yves Thibon]{jyt@univ-mlv.fr} 
\date{}

\begin{document}

\begin{abstract}The operad of moulds is realized in terms of an operational
calculus of formal integrals (continuous formal power series). This leads to
many simplifications and to the discovery of various suboperads.
In particular, we prove a conjecture of the first author about
the inverse image of non-crossing trees in the dendriform operad.
Finally, we explain a connection with the formalism of noncommutative symmetric
functions.
\end{abstract}

\keywords{Operads, Moulds, Trees, Noncommutative symmetric functions}

\maketitle

\section{Introduction}

A {\em mould}, as defined by Ecalle, is a ``function of a variable
number of variables'', that is, a sequence $f_n(u_1,\ldots,u_n)$ of
functions of $n$ (continuous or discrete) variables. He developed
around this notion a versatile formalism which is an essential
technical tool in his theory of resurgence \cite{eca1981,eca1992} and
in his later work on polyzetas \cite{eca2002,eca2003,eca2004} (see the
lecture notes \cite{Cresson} for an elementary introduction).

In \cite{Ch}, the first author constructed an operad ${\bf Mould}$
from the set of (rational) moulds, and identified several (old and
new) suboperads of it.

The aim of this article is to introduce an operational calculus on
formal integrals, which allows to simplify considerably the arguments
of \cite{Ch}, and also to obtain further results. 
In particular, we find that the operad Zinbiel introduced by Loday
\cite{Lod1995} is a sub-operad of ${\bf Mould}$. As Zinbiel is based
on permutations, this allows to consider the elements of the algebra
$\FQSym$ of free quasi-symmetric functions as moulds. For instance,
the classical Lie idempotents of Dynkin, Solomon and Klyachko give
interesting examples of alternal moulds.
We also find some other new suboperads, prove conjecture 5.7 of
\cite{Ch}, and obtain some new examples of moulds.

This article is a continuation of \cite{Ch}. In a few examples, we shall
assume that the reader is familar with the notation of \cite{NCSF2,NCSF6}.

\section{Moulds as nonlinear operators}

Let $\H$ be a vector space of formal integrals
\begin{equation}\label{formint}
h(t) = \int h_u t^{u-1} d\mu(u)\,,
\end{equation}
where  $h_u$ are homogeneous elements of degree $u$ in
some graded associative algebra ${\mathcal A}$. 
We will only need the cases where $\mu$ is the Lebesgue
measure on $\R$ or $\R_+$, or the discrete measure on $\N$,
which gives back power series in $t$ with noncommutative coefficients.
In these cases, the object (\ref{formint}) can be interpreted
as a linear map ${\mathcal V}\rightarrow{\mathcal A}$
on the vector space with basis $(t^u)$ for
$u$ in the support of $\mu$, and the theory is completely similar
to that of formal power series.

A mould $f=(f_n(u_1,\ldots,u_n))$ can be interpreted as a nonlinear operator
$F$ on $\H$, by setting
\begin{equation}
F[h]=\sum_{n\ge 0}\int\cdots\int f_n(u_1,\ldots,u_n)h_{u_1}\cdots
h_{u_n}t^{u_1+\cdots+u_n} d\mu(u_1)\cdots d\mu(u_n)\,.
\end{equation}
It will be convenient to set
\begin{equation}
H(t)=\int_0^t h(\tau)d\tau = \int h_u \frac{t^{u}}{u}d\mu(u)\,.
\end{equation}
We also define the {\em polarization} of $F$ as the collection
of multilinear operators (the $h^{(i)}$ are arbitrary elements of $\H$)
\begin{equation}
F_n[h^{(1)},\ldots,h^{(n)}]=\int\cdots\int f_n(u_1,\ldots,u_n)h^{(1)}_{u_1}\cdots
h^{(n)}_{u_n}t^{u_1+\cdots+u_n} d\mu(u_1)\cdots d\mu(u_n)\,.
\end{equation}

\section{Examples of moulds}

In this section, we translate all the examples of \cite{Ch} into the new
formalism, and provide some new ones.

\begin{example}{\rm The mould
\begin{equation}
f_n(u_1,\ldots,u_n)=\frac1{u_1\cdots u_n}
\end{equation}
corresponds to the operators
\begin{equation}
F_n[h^{(1)},\ldots,h^{(n)}]=H^{(1)}(t)\cdots H^{(n)}(t)\,.
\end{equation}
}
\end{example}

\begin{example}{\rm The time-ordered exponential
\begin{equation}\label{Udet}
U(t)= T\exp\left\{ \int_0^t h(\tau)d\tau \right\}\,,
\end{equation}
\ie, the unique solution of $U'(t)=U(t)h(t)$ with $U(0)=1$,
is given by the mould
\begin{equation}
f_n(u_1,\ldots,u_n)=\frac1{u_1(u_1+u_2)\cdots(u_1+u_2+\cdots+u_n)}\,.
\end{equation}

}
\end{example}

\begin{example}{\rm More generally, for a permutation $\sigma\in\SG_n$,
the mould
\begin{equation}
f_\sigma(u_1,\ldots,u_n)=\frac1{u_{\sigma(1)}(u_{\sigma(1)}+u_{\sigma(2)})\cdots
(u_{\sigma(1)}+u_{\sigma(2)}+\cdots+u_{\sigma(n)})}
\end{equation}
integrates over the simplex
$\Delta_\sigma(t)=\{0<t_{\sigma(1)}<t_{\sigma(2)} <\cdots <t_{\sigma(n)}<t\}$:
\begin{equation}
F_\sigma[h^{(1)},\ldots,h^{(n)}]=\int_{\Delta_\sigma(t)}h^{(1)}(t_1)\cdots
h^{(n)}(t_n)dt_1\cdots dt_n\,.
\end{equation}
It follows from the well-known decomposition of a product of simplices
as a union of simplices that these moulds form a subalgebra,
isomorphic to the algebra of free quasi-symmetric functions $\FQSym$, under the correspondence
$f_\sigma\mapsto \F_\sigma$ (cf. \cite{NCSF6}).

}
\end{example}

\begin{example}{\rm To each planar binary tree $T$, we can associate
an operator $F_T$ defined by $F_\bullet[h]=H$ and, if $T=T_1\wedge T_2$
has $T_1$ and $T_2$ as left and right subtrees
\begin{equation}
F_T[h]= \int_0^t F_{T_1}[h](\tau) h(\tau) F_{T_2}[h](\tau)d\tau\,.
\end{equation}
The kernels of these operators are the moulds associated to trees
in \cite{Ch}, which can be computed graphically as follows. All the leaves
of $T$ are labelled by  $1$, and the internal nodes are labelled by $t^{u_i-1}$,
in such a way that flattening the tree yields the $t^{u_i-1}$ in their natural
order
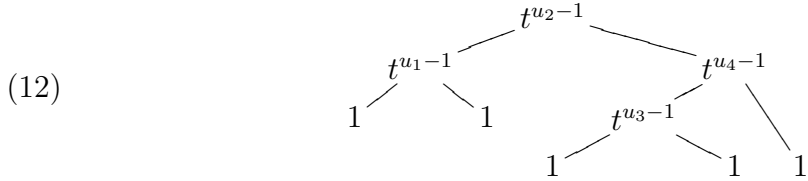
\entrymodifiers={+<4pt>}
\begin{equation}
\vcenter{\xymatrix@C=2mm@R=2mm{
*{}   &  *{} & *{} & {t^{u_2-1}}\ar@{-}[drr]\ar@{-}[dll] \\
*{}   & {t^{u_1-1}}\ar@{-}[dr]\ar@{-}[dl]  & *{} & *{}   & *{} &
{t^{u_4-1}}\ar@{-}[dl]\ar@{-}[ddr] \\
{1}   & *{}  & {1} & *{}   & {t^{u_3-1}}\ar@{-}[dl]\ar@{-}[dr]\\
*{}   & *{}  & *{} & {1}   & *{} & {1} & {1}\\
      }}
\end{equation}
The mould $f_T(u_1,\ldots,u_n)$ is obtained by evaluating the tree
according to the following rule: the outgoing flow of each node is
the integral $\int_0^t L(\tau) v(\tau) R(\tau)d\tau$, where $L(t)$
and $R(t)$ are the outputs of its left and right subtrees, and $v(t)$
its label. For example, the above tree evaluates to
\begin{equation}
\frac{t^{u_1+u_2+u_3+u_4}}{u_1 u_3(u_3+u_4)(u_1+u_2+u_3+u_4)}\,,
\end{equation}
as can be seen on the following picture
{
\entrymodifiers={+<4pt>}
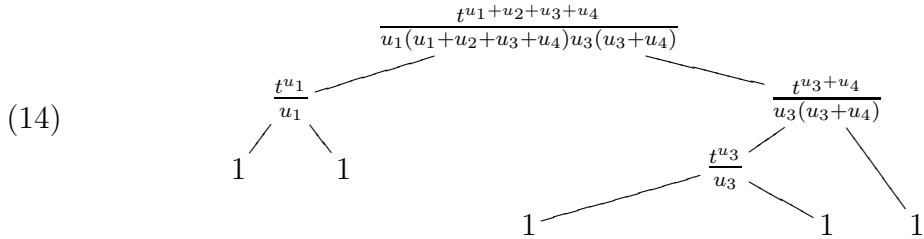
\begin{equation}
\vcenter{\xymatrix@C=2mm@R=2mm{
*{}   &  *{} & *{} &
{\frac{t^{u_1+u_2+u_3+u_4}}{u_1(u_1+u_2+u_3+u_4)u_3(u_3+u_4)}}
\ar@{-}[drr]\ar@{-}[dll] \\
*{}   & {\frac{t^{u_1}}{u_1}}\ar@{-}[dr]\ar@{-}[dl]  & *{} & *{}   & *{} &
{\frac{t^{u_3+u_4}}{u_3(u_3+u_4)}}\ar@{-}[dl]\ar@{-}[ddr] \\
{1}   & *{}  & {1} & *{}   & {\frac{t^{u_3}}{u_3}}\ar@{-}[dl]\ar@{-}[dr]\\
*{}   & *{}  & *{} & {1}   & *{} & {1} & {1}\\
      }}
\end{equation}
}
and the corresponding mould is obtained by setting $t=1$.
Each $F_T$ is a sum of operators $F_\sigma$ (sum over all
$\sigma$ such that the decreasing tree of $\sigma^{-1}$
has shape $T$).
This can be used as in \cite{HNT2} to derive the hook-length
formula for binary trees.
}
\end{example}

\begin{example}{\rm
The moulds associated to planar binary trees are related to the
solution of the quadratic differential equation
\begin{equation}
\frac{dx}{dt} = b(x,x),\ x(0)=1
\end{equation}
where the bilinear map $b(x,y)$ is assumed to have an integral
representation of the type
\begin{equation}
b(x,y)=\int x_u b_u y_u t^{u-1}d\mu(u)= (x*b*y)(t)
\end{equation}
the convolution $*$ being defined by
\begin{equation}
(x*y)(t)=\int x_u y_u t^{u-1}d\mu(u)
\end{equation}
and $b(t)=b(1,1)(t)$.
This can be recast in the form
\begin{equation}\label{eqbin}
x = 1 + B(x,x)
\end{equation}
where 
\begin{equation}
B(x,y)=\int_0^t b(x,y)(\tau)d\tau\,,
\end{equation}
so that
\begin{equation}\label{solbin}
x = 1 + B(1,1) + B(B(1,1),1)+ B(1,B(1,1))+ \cdots =  \sum_{T\in\CBT} B_T(1)
\end{equation}
where $\CBT$ is the set of (complete) binary trees, and for a tree $T$,
$B_T(a)$ is the result of evaluating the expression formed by labeling
by $a$ the leaves of $T$ and by $B$ its internal nodes.
Then, the term $B_T(1)$ in the binary tree solution is $F_T[b]$.
}
\end{example}

\begin{example}{\rm The mould \cite[(103)]{Ch}
\begin{equation}
y_{p,q}(u_1,\ldots,u_n)=\frac{u_p}{u_1\cdots u_n (u_1+\cdots+u_n)}
\end{equation}
(sum over all binary trees of type $(p,q)$) corresponds to the operator
\begin{equation}
Y_{p,q}[h^{(1)},\ldots,h^{(n)}]=\int_0^t H^{(1)}(\tau)\cdots H^{(p-1)}(\tau)
h^{(p)}(\tau)H^{(p+1)}(\tau)\cdots H^{(n)}(\tau)d\tau\,.
\end{equation}
}
\end{example}

\begin{example}{\rm The mould $TY$ defined by \cite[(104)]{Ch}
\begin{equation}
TY_n=\sum_{i=1}^n\alpha^{i-1} y_{i,n-i}
\end{equation}
corresponds to the operator
\begin{equation}
F[h]=\int_0^t (1-\alpha H(\tau))^{-1}h(\tau)(1-H(\tau))^{-1}d\tau\,.
\end{equation}
When $h$ is scalar (the $h_u$ commute), this reduces to
\begin{equation}
F[h]=\int_0^t (1-\alpha H(\tau))^{-1}(1-H(\tau))^{-1}dH(\tau)=
\frac{1}{1-\alpha}\log\left(\frac{1-\alpha H(t)}{1-H(t)}\right)\,.
\end{equation}
}
\end{example}

\begin{example}{\rm The mould \cite[(106)]{Ch}
\begin{equation}
\sum_{i=1}^n i y_{i,n-i}
\end{equation}
corresponds to the operator
\begin{equation}
F[h]=\int_0^t (1-H(\tau))^{-2}h(\tau)(1-H(\tau))^{-1}d\tau\,.
\end{equation}
When $h$ is scalar, this reduces to
\begin{equation}
F[h]=\int_0^t (1-H(\tau))^{-2}(1-H(\tau))^{-1}dH(\tau)=
\frac{H(t)(2-H(t))}{2(1-H(t))^2}\,.
\end{equation}
}
\end{example}

\begin{example}{\rm The following modified mould
\begin{equation}
\sum_{i=1}^n [i]_q y_{i,n-i},
\end{equation}
where $[i]_q$ is the quantum number $1+q+\dots+q^{i-1}$, corresponds to the operator
\begin{equation}
F[h]=\int_0^t (1-q H(\tau))^{-1}(1-H(\tau))^{-1}h(\tau)(1-H(\tau))^{-1}d\tau\,.
\end{equation}
When $h$ is scalar, this reduces to
\begin{equation}
F[h]=\int_0^t (1-q H(\tau))^{-1}(1-H(\tau))^{-2} dH(\tau)=
\frac{1}{1-q}\left(\frac{H(t)}{1-H(t)}-\frac{q}{1-q}\log\left(\frac{1-H(t)}{1-q
  H(t)}\right)\right)\,.
\end{equation}
}
\end{example}

\begin{example}{\rm The Connes-Moscovici series (\cite[(109)]{Ch})
is given by the mould
\begin{equation}
\frac{1}{n!}\sum_{k=1}^n (-1)^{n-k}{n\choose k} k  y_{k,n-k}
\end{equation}
and the corresponding operator is
\begin{equation}
F[h]=\int_0^t e^{H(\tau)}h(\tau)e^{-H(\tau)}d\tau\,,
\end{equation}
which reduces to $H(t)$ in the scalar case.
}
\end{example}

\begin{example}{\rm From the Solomon Lie idempotent, one can define
    the following mould
\begin{equation}\label{idsol}
\frac{1}{n}\sum_{\sigma\in \SG_n} (-1)^{d(\sigma)} {n-1 \choose d(\sigma)}^{-1} f_{\sigma}.
\end{equation}
Its output is the logarithm of $U(t)$ as defined by (\ref{Udet}).
It is also called the first Eulerian idempotent.

The $q$-deformation obtained in \cite{NCSF2}
yields a one-parameter family of moulds
\begin{equation}\label{qsol}
\frac{1}{n}\sum_{\sigma\in\SG_n} (-1)^{d(\sigma)} \qbin{n-1}{d(\sigma)}^{-1}
q^{\maj(\sigma) - {d(\sigma)+1 \choose 2} } f_{\sigma},
\end{equation}
where $\qbin{n-1}{d(\sigma)}$ is a quantum binomial coefficient.
}
\end{example}

\begin{example}{\rm The mould
\begin{equation}\label{idyn}
D_n=\sum_{i=0}^{n-1} (-1)^i \frac1{(u_1 u_{12} u_{123} \cdots u_{123..i} )
u_{1..n}
(u_{i+1..n} \cdots
 u_{n-1 n} u_n)} \end{equation}
(sum over trees of the form (left comb)$\wedge$(right comb))
corresponds to Dynkin's idempotent, more precisely
\begin{equation}
 D_n[h]=\int_{\Delta_n} [...[h(t_1),h(t_2)],h(t_3)]...,h(t_n)]d\mu(t_1)\cdots
d\mu(t_n)\,.
\end{equation}

}
\end{example}

\begin{example}{\rm The mould $PO$, defined in \cite[(113)]{Ch} by
\begin{equation}
PO_n=\frac1{u_1}\prod_{i=2}^n\frac{u_1+\cdots+u_{i-1}+qu_i}{u_i(u_1+\cdots+u_i)}
\end{equation}
can be decomposed on the permutations $f_\sigma$ as
\begin{equation}
PO_n=\sum_{\sigma\in\SG_n}q^{s(\sigma^{-1})-1}f_\sigma\,,
\end{equation}
where $s(\sigma)$ is the number of saillances of $\sigma$, \ie, the
number of $i$ such that $\sigma_i>\sigma_j$ for all $j<i$. This statistics
has the same distribution as the number of cycles.

}
\end{example}

\begin{example}{\rm A mould is {\em alternal} if and only if it satisfies
\begin{equation}
F[h_1+h_2]=F[h_1]+F[h_2]
\end{equation}
whenever $h_1$ and $h_2$ commute.
Typically, $\H$ is a Lie algebra
and $F$ takes its values in a completion of $U(\H)$. Then, $F$
is alternal if and only if it preserves primitive elements. 
For example, (\ref{idsol}), (\ref{qsol}) and (\ref{idyn}) are alternal.
Similarly, $F$
is {\em symmetral} if it maps primitive elements to group-like elements.
Otherwise said, 
\begin{equation}
F[h_1+h_2]=F[h_1]\cdot F[h_2]
\end{equation}
as soon as $h_1$ and $h_2$ commute.
}
\end{example}

\begin{example}{\rm The dendriform products $\prec$ and $\succ$
are given by
\begin{equation}
(F\succ G)[h]=\int_0^t F[h](\tau)\cdot \frac{d}{d\tau}G[h](\tau)d\tau\,,\quad
(F\prec G)[h]=\int_0^t \frac{d}{d\tau} F[h](\tau)\cdot G[h](\tau)d\tau\,.
\end{equation}
On permutational moulds $f_{\sigma}$, these coincide with the half-shifted
shuffles, {\it e.g.}, $312\prec 12= (31\shuffle 45)\cdot 2$.

}
\end{example}

\begin{example}{\rm The preLie product $F\sous G=F\succ G-G\prec F$
is given by
\begin{equation}
F\sous G [h]=\int_0^t[ F[h],G'[h]](\tau)d\tau\,,
\end{equation}
where $G'[h]$ denotes the derivative with respect to $\tau$.
On this expression it is clear that if $h$ is primitive, so is
$F\sous G [h]$ if $F$ and $G$ are alternal.

}
\end{example}

\section{Operadic operations on operators}

The $i$th operadic composition of two homogeneous moulds
$f_m$ and $g_n$,
as defined in\cite{Ch},
 corresponds to the operator whose polarization is
\begin{multline}
F_m\circ_i G_n
[h^{(1)},\ldots,h^{(i-1)};
h^{(i)},\ldots,h^{(i+n-1)};
h^{(i+n)},\ldots,h^{(m+n-1)}]\\
=
F_m[h^{(1)},\ldots,h^{(i-1)};
\frac{d}{dt}G_n[h^{(i)},\ldots,h^{(i+n-1)}];
h^{(i+n)},\ldots,h^{(m+n-1)}]\,.
\end{multline}
It follows from this description that the linear span of the
$F_\sigma$ is stable under these operations, hence form a suboperad.

\begin{example}{\rm According to the definition of \cite{Ch}, 
\begin{equation}
f_{312}\circ_2 f_{12}=
{\frac {1}{u_{{4}} \left( u_{{4}}+u_{{1}} \right)  \left(
u_{{2}}+u_{{4}}+u_{{3}}+u_{{1}} \right) u_{{2}}}}
=f_{2413}+f_{4213}+f_{4123}\,.
\end{equation}
and
\begin{equation}
F_{312}\circ_2 F_{12}[h^{(1)},h^{(2)},h^{(3)},h^{(4)}]
=
F_{312}[h^{(1)},\frac{d}{dt}F_{12}[h^{(2)},h^{(3)}],h^{(4)}]\,,
\end{equation}
where
\begin{equation}
F_{12}[h^{(2)},h^{(3)}]=\int_0^tdt_3\int_0^{t_3}dt_2 h^{(2)}(t_2) h^{(3)}(t_3)
\end{equation}
has as derivative, evaluated at $t_3$
\begin{equation}
F_{12}[h^{(2)},h^{(3)}]'(t_3)=\int_0^{t_3}dt_2 h^{(2)}(t_2) h^{(3)}(t_3)\,.
\end{equation}
When plugged into $F_{312}$, with the shifts $312\rightarrow 413$, this
yields
\begin{equation}
\int_{ \Delta_{413;2}(t,t_3)}h^{(1)}(t_1) h^{(2)}(t_2) h^{(3)}(t_3)h^{(4)}(t_4)
dt_1dt_2dt_3dt_4
\end{equation}
where the integration domain decomposes as
\begin{equation}
\Delta_{413;2}(t,t_3):=\{0<t_4<t_1<t_3<t;0<t_2<t_3\}
=\Delta_{2413}(t)\cup \Delta_{4213}(t)\cup \Delta_{4123}(t)\,,
\end{equation}
as expected.

}
\end{example}
To give the general rule, it is sufficient to compute 
\begin{equation}
    F_{\Id_m} \circ_i F_{\Id_n} = \sum F_\sigma
\end{equation}
where the sum is over permutations $\sigma$ in the shuffle
\begin{equation}
    ((1,\dots, i-1) \shuffle (i,\dots, i+n-2)) \cdot (i+n-1, \dots, m+n-1))\,.
\end{equation}
\begin{example}{\rm
\begin{gather}
  F_{123} \circ_1 F_{123} = F_{12345} \\
  F_{123} \circ_2 F_{123} = F_{12345} + F_{21345} + F_{23145} \\
  F_{123} \circ_3 F_{123} = F_{12345} + F_{13245} + F_{13425} + F_{31245} +
                           F_{31425} + F_{34125}
\end{gather}
}
\end{example}
This operad is anticyclic. It is in fact isomorphic to Zinbiel
(cf.  \cite{Lod}), up to mirror image of permutations.
The action of the $(n+1)$-cycle $\gamma$ on a homogenous
mould $f(u_1,\ldots,u_n)$ of degree $n$ is defined by
\begin{equation}
\gamma f(u_1,\ldots,u_n)=f(u_2,u_3,\ldots,u_n,-u_1-u_2-\cdots -u_n)\,.
\end{equation}
The subspace spanned by permutational moulds $f_\sigma$ is stable
under the action of $\gamma$. Explicitly,
\begin{equation}
\gamma f_\sigma = (-1)^{|v|}\sum_{\tau\in u\shuffle v}f_\tau
\end{equation}
where the words $u$ and $v$ are defined as follows. Let
$\sigma'(i)=\sigma(i)+1 \mod n$,  write $\sigma'=u1w$
and $v=1\bar w$ (where $\bar w$ is the mirror image of $w$).

This follows easily from the product formula for permutational moulds.
For example,
\begin{gather}
\gamma f_{1432}=-f_{2134}-f_{1234}-f_{1324}-f_{1342}= -f_2 f_{134}\\
\gamma
f_{2143}=f_{3214}+f_{3124}+f_{3142}+f_{1342}+f_{1324}+f_{1432}=f_{32}f_{14}\,.
\end{gather}

\begin{example}{\rm The operadic preLie product $F\circ G$ is
\begin{equation}
(F\circ G)[h]=\sum_{i=1}^m (F\circ_i G)[h]= DF[h](G'[h])\,,
\end{equation}
that is, the differential $DF[h]$ of $F$ at the point $h$, evaluated
on the vector $G'[h]$, where, as above, $G'[h]$ denotes the $t$-derivative.
On this description, it is clear that $\circ$ preserves alternality.

}
\end{example}

\begin{example}{\rm The derivation $\partial$ of \cite[(85)]{Ch} is
\begin{equation}
(\partial F)[h]= DF[h](1) := \lim_{\varepsilon\rightarrow
0}\frac{F[h+\varepsilon]-F[h]}{\varepsilon}\,,
\end{equation}
the derivative of $F$ at $h$ in the direction on the constant
function $1$. On this description,
it is easy to check that $\partial$ is a derivation for the various
products. For example,
\begin{equation}
\partial(F\succ G)[h]=\int_0^t \{DF[h](1) G'[h] + F[h]DG'[h](1)\}d\tau
=(\partial F\succ G + F\succ\partial G)[h]\,.
\end{equation}
For those moulds such that $F[h]$ reduces to an analytic function $F(H)$ of
$H$ in the scalar case, $\partial F[h]$ reduces to the derivative of $F(H)$
with respect to $H$.

}
\end{example}

\begin{example}{\rm The over and under operations are given by
\begin{equation}
(F/G)[h^{(1)},\ldots,h^{(m+n)}]=G[F[h^{(1)},\ldots,h^{(n)}]h^{(n+1)},h^{(n+2)},\ldots,h^{(m+n)}]
\end{equation}
and
\begin{equation}
(F\backslash G)[h^{(1)},\ldots,h^{(m+n)}]
=
F[h^{(1)},\ldots,h^{(n-1)},h^{(n)}G[h^{(n+1)},\ldots,h^{(m+n)}]].
\end{equation}

}
\end{example}

\begin{example}
{\rm  
The ARIT map is
\begin{equation}
{\rm ARIT}(F,G)[h]=DF[h](G[h] h-h G[h]).  
\end{equation}

The ARI map is
\begin{equation}
{\rm ARI}(F,G)[h]=DF[h](G[h] h-h G[h])-DG[h](F[h] h-h F[h])+F[h]G[h]-G[h]F[h].
\end{equation}
}
\end{example}

\section{Non-crossing trees and non-interleaving forests}

In \cite{Ch}, the first author has constructed an operad
on the set of non-crossing trees, and formulated a
conjecture about the inverse image of non-crossing
trees in the dendriform operad. In this section, we prove
this conjecture by means of a new presentation of this
operad.
The reader is referred to \cite{Ch} for the background
on non-crossing trees.

\subsection{A bijection}

\begin{definition}
A non-interleaving forest is a labeled rooted forest such that
the set of labels of any subtree is an interval.
\end{definition}
In particular, the labels of each connected component is an interval.
A non-interleaving tree is a non-interleaving forest with a single
component. Our new presentation of NCT will be based on non-interleaving
trees.

Let $T$ be a non-crossing tree. We define a poset $P$ from $T$ as follows.
Fist, label each diagonal edge of $T$ by the number of the unique open side
which it separates from the base, and each side edge by its own number.
Then, set $i<_P j$ iff the edge $i$ is separated from the base by the edge $j$.

\begin{figure}
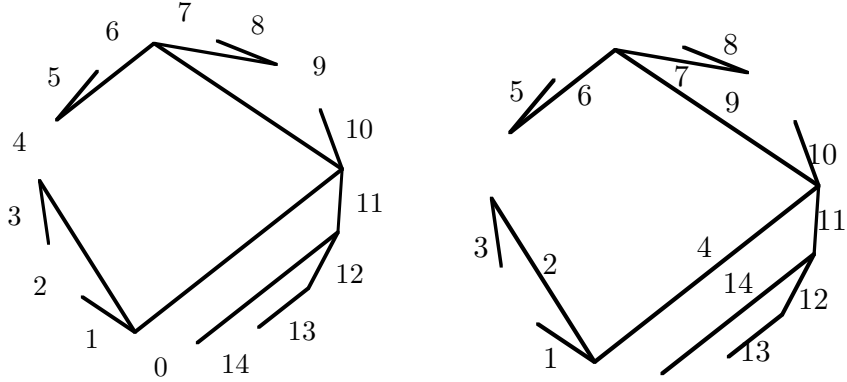

\begin{center}
\includegraphics[height=5cm]{mor.1}
\hskip1cm
\includegraphics[width=5cm]{mor.2}
\caption{\label{nonct}A non-crossing tree, and the corresponding labeling}
\end{center}
\end{figure}
\begin{lemma}If $P$ is constructed from a non-crossing tree $T$ by
the above process, its Hasse diagram $F$ is a non-interleaving forest.
Moreover, the correspondence $T\mapsto F$ is a bijection between
non-crossing trees and non-interleaving forests.
\end{lemma}

\begin{proof}
The roots of the trees are the labels of the edges having the
base on their external sides. The edges $\alpha$ which are on the other sides
of the root edges are labeled by disjoint intervals of $[1,n]$,
and these intervals are the labels of the edges which are separated
from the bases by those $\alpha$.
Conversely, to each vertex $v$ of a non-interleaving forest, on can
associate an edge from the left side of $\min\{k|k<_P v\}$ to
the right side of $\max\{k|k<_P v\}$. This yields a non-crossing
tree mapped to $P$ by the previous algorithm. Hence, the correspondence
is onto.
Finally, non-crossing trees and non-interleaving forests have the same
grammar, hence in particular the same generating series. 
\end{proof}

For example, the non-interleaving forest associated to the non-crossing
tree on Figure \ref{nonct} is

\entrymodifiers={+<4pt>}
\begin{equation}
\vcenter{\xymatrix@C=2mm@R=2mm{
*{} & *{} & *{} & {4}\ar@{-}[drr]\ar@{-}[dll]\ar@{-}[d]
& *{} & *{} & *{} & {11} & *{} & {14}\ar@{-}[dl]\ar@{-}[dr] & *{} \\
*{} & {2}\ar@{-}[dl]\ar@{-}[dr] & *{} & {6}\ar@{-}[d]
& *{} & {9}\ar@{-}[dl]\ar@{-}[dr] & *{} & *{} & {12} & *{} & {13} \\
{1} & *{} & {3} & {5} & {7}\ar@{-}[d] & *{} & {10} & *{} & *{} & *{} & *{} \\
*{} & *{} & *{} & *{} & {8} & *{} & *{} & *{} & *{} & *{} & *{} \\
      }}
\end{equation}

\subsection{Associated rational functions}

In \cite{Ch}, one associates a rational function $f_T$ to
a non-crossing tree by the following rule:
\begin{equation}
f_T=\prod_{e\in E(T)}\frac1{{\rm ev}(e)}
\end{equation}
where $E(T)$ is the set of edges of $T$, and the evaluation
of an edge is given by
\begin{equation}
{\rm ev}(e)=\sum_{i}u_i\,,
\end{equation}
where $i$ runs over the labels of the edges separated from 0 by $e$.

It follows from the above arguments that
\begin{equation}
f_T=\prod_{i=1}^n\frac1{\sum_{ j{\le}_P i} u_j}\,.
\end{equation}

\begin{lemma}The fraction associated to $T$ is the
sum of the linear extensions of $P$:
\begin{equation}
f_T=\sum_{\sigma\in L(P)}f_\sigma\,.
\end{equation}
\end{lemma}

\begin{proof}
$f_T$ is in $\FQSym$, since NCT is a suboperad of Dend,
\begin{equation}
f_T=\sum c_\sigma(T)f_\sigma\,,
\end{equation}
and $c_\sigma(T)$ is the iterated residue of $f_TT$ at 
$x_{\sigma_1}=0$, $x_{\sigma_2}=0,\ldots$
so that $c_\sigma(T)=0$ if $\sigma\not\in L(P)$, and
$c_\sigma(T)=1$ otherwise.
\end{proof}

\subsection{Proof of the conjecture}
Hence, the morphism form the free NCT-algebra on one generator to
the free dendriform algebra on one generator $\PBT$, regarded
as a subalgebra of $\FQSym$, consists in mapping a non-interleaving
forest on the sum of its linear extensions:
\begin{equation}
T\mapsto P\mapsto F\mapsto  \sum_{\sigma\in L(P)}\F_\sigma= \sum_{t\in I}\P_t\,,
\end{equation}
where $\P_t$ is the natural basis of $\PBT$, and $I$ a set of binary trees.
Conjecture 6.5 of \cite{Ch} is the following:

\begin{theorem}
$I$ is an interval of the Tamari order.
\end{theorem}

\begin{proof} 
Under this morphism, a forest becomes the product of its connected
components. It is therefore sufficient to prove that the linear
extensions of a non-interleaving tree have the required properties.
The linear extensions of a tree are computed recursively by shuffling
the linear extensions of the subtrees of the root and concatenating the
root at the end. By definition of a non-interleaving tree,
these form an interval of the permutohedron, whose minimum avoids
the pattern 312 and maximum avoids 132. This is a known characterization
of Tamari intervals.
\end{proof}

\subsection{Another approach}
One can also start with an operad NIT defined directly 
on non-interleaving trees.

There are two natural binary operations on non-interleaving trees.
Let $T_1$ and $T_2$ be two such trees. 
Let $k$ be the number of vertices of $T_1$, and denote by $T'_2$
the result of shifting the labels of $T_2$ by $k$.
Define
\begin{itemize}
\item $T_1\prec T_2 =$ grafting of the root of $T'_2$ on the root of $T_1$
\item $T_1\succ T_2 =$ grafting of the root of $T_1$ on the root of $T'_2$
\end{itemize}
These are two magmatic operations, satisfying the single relation
\begin{equation}\label{relnit}
(x\succ y)\prec z = x\succ(y\prec z)\,.
\end{equation}
The operad defined by this relation has been first considered in
\cite{leroux} under the name of L-algebra.

Consider now the free algebra on one generator. Its monomials can
be represented by bicolored complete binary trees, whose internal vertices
are colored by $\prec$ or $\succ$. Relation (\ref{relnit}) implies
that a basis is formed by the trees having no right edge from
a vertex $\succ$ to a vertex $\prec$.
Loday \cite{LodLoth} has presented a general method (relying
on Koszul duality for quadratic operads) for counting such 
$k$ colored binary trees avoiding a set $Y$ of edges.
Their generating series (with alternating signs) $g(t)$ is obtained
by inverting (for the composition of power series) the series
$f(t)=-t+kt^2-|X_2|t^3+|X_3|t^4-\cdots$, where $X_n$ is the
set of trees with $n$ internal nodes whose all edges are in $Y$. 
Here $k=2$, and there is only one tree, with two internal vertices,
having all edges in $Y$. Hence,
$f(t)=-t+2t^2-t^3$, and we get the sequence A006013 of \cite{Sloane}
$g(t)=-t+2t^2-7t^3+30t^4-143t^5+...$ (based non-crossing trees).

\section{Appendix: Moulds over the positive integers}

When the variables $u_k$ take only positive integer
values, we denote them by $i_k$ and write
$f_I=f_{i_1,\ldots,i_r}$ instead of $f(i_1,\ldots,i_r)$.
This corresponds to the choice
\begin{equation}
d\mu(t)=\sum_{n\ge 1}\delta(t-n)\,.
\end{equation}
In this case, there is a close connection with the formalism
of noncommutative symmetric functions, which can also represent
nonlinear operators on powers series with noncommuting coefficients.

In this appendix, we will give the interpretation of some of the
previous examples in this context, as well as of some new ones.
We assume here that the reader is familiar with the notation of \cite{NCSF2}.

\subsection{Generating sequences of noncommutative symmetric functions}

By definition, $\Sym$ is a graded free associative algebra, with exactly
one generator for each degree. Several sequences of generators are of common
use, some of which being composed of primitive elements, whilst other
are sequences of divided powers, so that their generating series is
group-like. Each pair of such sequences $(U_n)$, $(V_n)$ defines two moulds,
whose coefficients express the expansions of the $V_n$ on the $U^I$,
and vice-versa. Ecalle's four fundamental symmetries reflect the four
possible combinations of the primitive or group-like characteristics.

If we denote by $\LL$ the (completed) primitive Lie algebra of $\Sym$ and
by $\GG=\exp\,\LL$ the associated  multiplicative group,
we have the following table
$$
\begin{tabular}{|c|c|}
\hline
$\LL \rightarrow \LL$ & \text{Alternal}\\
\hline
$\LL \rightarrow \GG$ & \text{Symmetral}\\
\hline
$\GG \rightarrow \LL$ & \text{Alternel}\\
\hline
$\GG \rightarrow \GG$ & \text{Symmetrel}\\
\hline
\end{tabular} 
$$
The characterization of alternal moulds in terms of shuffles
is equivalent to Ree's theorem (cf. \cite{Reut}): {\em the orthogonal of
the free Lie algebra in the dual of the free associative algebra
is spanned by proper shuffles.}

The composition of moulds is the usual composition of the
corresponding operators. Since the relationship between
two sequences of generators 
of the same type (divided powers or grouplike)
can always be written in the form
\begin{equation}
V_n (A) = U_n(XA)\quad (\text{or } V(t)=U(t)*\sigma_1(XA))\,,
\end{equation}
where $X$ is a virtual alphabet (commutative and ordered, \ie,
a specialization of $QSym$), the composition of alternal or symmetrel 
moulds can also be
expressed by means of the internal product.

\subsection{$S_n$ and $\Lambda_n$: symmetrel}
The simplest example just gives the coefficients of the
inverse of a generic series regarded as $\lambda_{-t}(A)$.
It is a symmetrel mould:
\begin{equation}
S_n = \sum_{I\vDash n}f_I\Lambda^I\,,\quad\ f_I=(-1)^{n-l(I)}\,.
\end{equation}

\subsection{$S$ and $\Psi$: symmetral/alternel} 

The mould
\begin{equation}
f_I=\frac1{i_1(i_1+i_2)\ldots (i_1+\ldots i_r)}
\end{equation}
gives the expression of $S_n$ over $\Psi^I$:
\begin{equation}
S_n =\sum_{I\vDash n}f_I \Psi^I\,.
\end{equation}
Since $\sigma'(t)= \sigma(t)\, \psi(t)$, this expresses the
solution of the differential equation in terms of iterated integrals 
\begin{multline}\label{INTIT}
\sigma(t)= 1 + \int_0^t \! dt_1\, \psi(t_1) +
\int_0^t \! dt_1 \int_0^{t_1} \! dt_2 \, \psi(t_2)\psi(t_1)\\
 +
\int_0^t \! dt_1 \int_0^{t_1} \! dt_2 \int_0^{t_2} \! dt_3\,
\psi(t_3)\psi(t_2)\psi(t_1) + \cdots\\
= T\exp\left\{\int_0^t\psi(s)ds\right\}\,,
\end{multline}
or as Dyson's $T$-exponential.

\subsection{An alternal mould: the Magnus expansion}

The expansion of $\Psi_n$ in the basis $(\Phi^K)$ is given by
\begin{equation}
\Psi_n = \sum_{|K|=n}
\left[
\sum_{i=1}^{\ell (K)} (-1)^{i-1}{\ell (K)-1 \choose i-1} k_i
\right]
{\Phi^K \over \ell (K)! \pi (K)}
\ ,
\end{equation}
where $\pi(K)=k_1\cdots k_r$.
Using the symbolic notation
\begin{equation}
\{\Phi_{i_1}\cdots\Phi_{i_r}\, ,\, F\}
=\ad\Phi_{i_1} \ad\Phi_{i_2}\cdots \ad\Phi_{i_r} (F)
=[\Phi_{i_1},[\Phi_{i_2},[\ldots[\Phi_{i_r}\, ,\, F]\ldots]]]
\end{equation}
and the classical identity
\begin{equation}
e^a b e^{-a} =\sum_{n\ge 0} {(\ad a)^n\over n!}\, b = \{e^a\, ,\, b\} \ ,
\end{equation}
we obtain
\begin{equation}
\psi(t) =\sum_{n\ge 0} {(-1)^n\over (n+1)!}\{\Phi(t)^n\, ,\,\Phi'(t)\}
=\left\{ {1-e^{-\Phi(t)} \over \Phi(t)} \, ,\, \Phi'(t) \right\}
\end{equation}
which by inversion gives the Magnus formula:
\begin{equation}\label{MAGNUS}
\Phi'(t)= \left\{ {\Phi(t)\over 1-e^{-\Phi(t)}  } \, ,\, \psi(t)
\right\}
= \sum_{n\ge 0} {B_n\over n!} (\ad \Phi(t))^n\, \psi(t)
\end{equation}
the $B_n$ being the Bernoulli  numbers.

\subsection{Another alternal mould: the continuous BCH expansion}

The expansion of
$\Phi(t)$ in the basis $(\Psi^I)$ is given by the series
\begin{equation}\label{BCHCP}
\Phi(t)= \sum_{r\ge 1}\int_0^t dt_1\cdots \int_0^{t_{r-1}}dt_r
\sum_{\sigma\in \S_r}
{(-1)^{d(\sigma)} \over r} {r-1\choose d(\sigma)}^{-1}
\psi(t_{\sigma(r)})\cdots \psi(t_{\sigma(1)}) \ .
\end{equation}
Thus, the coefficient of $\Psi^I=\Psi_{i_1}\cdots\Psi_{i_r}$ in the
expansion of $\Phi_n$ is equal to
\begin{equation}
n\int_0^1 dt_1\cdots \int_0^{t_{r-1}} dt_r
\sum_{\sigma\in\S_r} {(-1)^{d(\sigma)}\over r}
{r-1\choose d(\sigma)}^{-1}
t_{\sigma(r)}^{i_1-1}\cdots t_{\sigma(1)}^{i_r-1} \ .
\end{equation}

It is worth observing that this expansion, together with a simple
expression of $\Psi_n$ in terms of the dendriform operations
of $\FQSym$, recently led Ebrahimi-Fard, Manchon, and Patras \cite{emp}, to an
explicit solution of the Bogoliubov recursion for renormalization
in Quantum Field Theory.

\subsection{Moulds related to the Fer-Zassenhaus series}

The noncommutative power sums of the
third kind $Z_n$ are defined by
\begin{equation}
\sigma(A;t) = \exp(Z_1\, t) \, \exp({Z_2 \over 2}\, t^2) \,
\dots \, \exp({Z_n \over n}\, t^n) \, \dots
\end{equation}
The first values of $Z_n$ are
$$
Z_1 = \Psi_1 \ , \ Z_2 = \Psi_2 \ , \
Z_3 = \Psi_3 + {1 \over 2} \, [\Psi_2,\Psi_1] \ ,
$$
$$
Z_4 = \Psi_4 + {1 \over 3}\, [\Psi_3,\Psi_1] +
{1\over 6} \, [[\Psi_2,\Psi_1],\Psi_1] \ ,
$$
$$
Z_5 = \Psi_5 + {1 \over 4}\, [\Psi_4,\Psi_1] + {1\over 3} \, [\Psi_3,\Psi_2]
+ {1 \over 12} \, [[\Psi_3,\Psi_1],\Psi_1]
$$
$$
- {7 \over 24} \, [\Psi_2,[\Psi_2,\Psi_1]]
+ {1 \over 24} \, [[[\Psi_2,\Psi_1],\Psi_1],\Psi_1] \ .
$$
This defines  interesting alternal moulds. There is no
known expression for $Z_n$ on the $\Psi^I$, but
Goldberg's explicit formula  (see \cite{Reut}) for the
Hausdorff series gives the decomposition of $\Phi_n$ on the basis $Z^I$.

The fact that $Z_n$ is a Lie series is known as the Fer-Zassenhaus ``formula''.

\subsection{A one-parameter family}

It follows from the characterization of Lie idempotents
in the descent algebra that
$P_n(A;q)=(1-q^n) \Psi_n\left({A\over 1-q}\right)$ is a noncommutative
power sum.  
The corresponding Lie idempotent is
\begin{equation}
\varphi_n(q) =
{1\over n}\sum_{|I|=n}
{(-1)^{d(\sigma)}\over \qbin{n-1}{d(\sigma)}}
q^{\maj(\sigma) - {d(\sigma)+1 \choose 2} } \sigma
\end{equation}
It specializes to
\begin{equation}
\varphi_n(0)=\theta_n\,,\qquad
\varphi_n(1) = \phi_n\,\qquad
\varphi_n(\omega) = \kappa_n\,
\end{equation}
where $\omega$ is a primitive $n$th root of unity (and $\varphi_n(\infty) = \theta_n^*$).

The nonlinear operator $E_q[h(t)]$, where  $h(t)=\sum_{n\ge 1}H_nt^{n-1}$, is

\begin{equation}
E_q[h(t)]= \sum_I c_I(q) H^I t^{|I|}
\end{equation}
Then,
\begin{equation}
E_1[h(t)]=\exp\int_0^th(s)ds=\exp H(t)
\end{equation}
while $E_0$ is Dyson's chronological exponential
\begin{equation}
E_0[h(t)]=T\exp\int_0^th(s)ds
=1+\int_0^t dt_1 h(t_1) + \int_0^t dt_1\int_0^{t_1} dt_2 h(t_2)h(t_1)
+\cdots
\end{equation}

\subsection{Another one-parameter family}

In \cite{NCSF2}, it is proved that there exists a unique
sequence $\pi_n(q)$ of Lie idempotents which are left and right eigenvectors
of $\sigma_1((1-q)A)$ for the internal product:
\begin{equation}
\sigma_1((1-q)A) * \pi_n(q) = \pi_n(q) * \sigma_1((1-q)A) = (1-q^n)\pi_n(q)
\end{equation}
These elements have the following specializations:
\begin{equation}
\pi_n(1) = {\Psi_n \over n}, \frac1nK_n(\zeta),\ \pi_n(0) =\frac1n Z_n\,. 
\end{equation}
In particular, the associated alternal moulds provide an interpolation
between the $T$-exponential and the Fer-Zassenhaus expansion.

\small
\section*{Acknowledgements}
This work has been partially supported by Agence Nationale de la Recherche,
 grant ANR-06-BLAN-0380

\end{document}